\documentclass[12pt, twoside, a4paper]{article}

\usepackage{latexsym}
\usepackage{amscd}
\usepackage{theorem}
\usepackage{pifont}
\usepackage{amsfonts}
\usepackage{xspace}
\usepackage{amssymb}
\usepackage{fancyhdr}
\usepackage{mathrsfs}
\usepackage{amsmath}%
\usepackage{graphics}%
\usepackage{graphicx}%
\usepackage{euscript}%
\usepackage{psfrag}%
\usepackage{empheq}%
\usepackage{upgreek}
\usepackage{eufrak}

\DeclareMathAlphabet{\mathpzc}{OT1}{pzc}{m}{it}

{\theorembodyfont{\slshape} \newtheorem{theorem}{Theorem}[section]}
{\theorembodyfont{\slshape} \newtheorem{lemma}[theorem]{Lemma}}
{\theorembodyfont{\slshape} \newtheorem{proposition}[theorem]{Proposition}}
{\theorembodyfont{\slshape} }
{\theorembodyfont{\slshape} }
{\theorembodyfont{\slshape} }

\DeclareRobustCommand{\rchi}{{\mathpalette\irchi\relax}}
\newcommand{\irchi}[2]{\raisebox{\depth}{$#1\chi$}} 
\newenvironment{proof}{\noindent\textbf{Proof:\ }}{$\hfill{\bullet}$}

\numberwithin{equation}{section}

\title{\textsc{Dirichlet eigenvalues of the Laplacian on full one-sided shift space}}                

\author{Shrihari Sridharan\footnote{Indian Institute of Science Education and Research Thiruvananthapuram (IISER-TVM), Maruthamala P.O., Vithura, Thiruvananthapuram, INDIA. PIN 695 551.} \\ {\tt shrihari@iisertvm.ac.in}  \bigskip \\ Sharvari Neetin Tikekar\footnote{Indian Institute of Science Education and Research Thiruvananthapuram (IISER-TVM), Maruthamala P.O., Vithura, Thiruvananthapuram, INDIA. PIN 695 551.} \\ {\tt sharvai.tikekar14@iisertvm.ac.in}} 

\date{\today}

\begin{document}

\maketitle
\thispagestyle{empty}

\bigskip 

\begin{abstract} 
\noindent
The full one sided shift space over finite symbols is approximated by an increasing sequence of finite subsets of the space. The Laplacian on the space is then defined as a renormalised limit of the difference operators defined on these subsets. In this work, we determine the spectrum of these difference operators completely, using the method of spectral decimation. Further, we prove that under certain conditions, the renormalised eigenvalues of the difference operators converge to an eigenvalue of the Laplacian. 
\end{abstract}
\bigskip \bigskip

\begin{tabular}{l c l}
\textbf{Keywords} & : & Laplacian on the shift space; \\ & & Spectral decimation; \\ & & Dirichlet spectrum. 
\\
\textbf{AMS Subject Classifications} & : & 47A75, 39A70, 37B10.
\end{tabular}
\bigskip 
\bigskip

\newpage
\section{Introduction}

\noindent 
Symbolic dynamics is a branch of study in mathematics that began, in \cite{hadamard}, as a means to understand geodesic flows in negatively curved surfaces. However, over the years, this branch of mathematics has proved itself to be useful to understand various dynamical systems, for example, rational maps restricted on their respective Julia sets \cite{lyubich}, Axiom A diffeomorphisms restricted on a basic set \cite{bowen1, bowen2}, expanding maps restricted on a compact subset of a Riemannian manifold \cite{bowen3}, {\it etc.} Apart from being useful in the modelling of various dynamical systems, this branch of mathematics is, in itself, an interesting abstract space to work with, thus captivating the minds of mathematicians and physicists alike. For detailed discussion on the topic, see \cite{MH, bks, bmn, kitchens, parrypolli, LM}, {\it etc.} 
\medskip 

\noindent 
In \cite{arxiv}, the authors took the first step towards investigating certain analytical aspects of the symbolic space, from the point of view of calculus, owing to the topology of the space. The symbolic space over $N (> 1)$ symbols is the space of one sided sequences, defined as 
\[ \Sigma_{N}^{+}\ \ :=\ \ \bigg\{ x = (x_{1}\, x_{2}\, \cdots)\ :\ x_{i} \in \left\{ 1, 2, \cdots, N \right\} \bigg\}. \] 
The topology on $\Sigma_{N}^{+}$ is dictated by the metric that separates two points by taking into consideration, where the two points begin to disagree. The authors in \cite{arxiv} build a Laplacian on $\Sigma_{N}^{+}$, as a scalar limit of difference operators $\{ H_{m} \}_{m\, \ge\, 0}$ defined on certain nested sequence $\{ V_{m}\}_{m\, \ge\, 0}$ of finite subsets of $\Sigma_{N}^{+}$. This Laplacian resembles the discrete approximation of the Laplacian in the classical setting. The effective resistance metric, that arises due to the energy form in Kigami's analysis on post-critically finite self-similar sets \cite{kigamimetric}, has no significance in this construction of the Laplacian, as it turns out to be discrete, as proved in \cite{arxiv}. The authors further solve the Dirichlet boundary value problem in the symbolic settings, proving the existence of the solution and its uniqueness upto harmonic functions.
\medskip

\noindent
A natural question that arises now, is to determine the complete Dirichlet spectrum of the Laplacian on $\Sigma_{N}^{+}$. In this paper, we focus on determining a particular type of the eigenvalues of the Laplacian, which are obtained as a renormalised limit of the eigenvalues of $H_{m}$. The structure of the eigenvalues of $H_{m}$ can be studied through a process called decimation, as designed in \cite{rammal, ramtoul}. The spectral decimation method was then employed successfully by Shima in \cite{shima}, Fukushima and Shima in \cite{fuku_shima}, to study the eigenvalues of the Laplacian on the Sierpinski gasket.
\medskip 

\noindent 
Our main goal in the present paper, is to determine the eigenvalues of $H_{m}$ under the Dirichlet boundary conditions and their multiplicities explicitly. We also prove that a particular sequence $\{ \lambda_{m} \}_{m\, \ge\, 1}$ of Dirichlet eigenvalues of the sequence of difference operators $\{ H_{m} \}$ with some renormalisation factor, converges to a Dirichlet eigenvalue of the Laplacian $\Delta$. However, we wish to point out that the question whether these eigenvalues constitute the entire spectrum of the Laplacian, still remains open.   
\medskip

\noindent
The paper is organised as follows. In section \eqref{preliminaries}, we present the basic definitions of the difference operators and the Laplacian on the full one sided shift space. In this paper, we extensively use the notations used in \cite{arxiv}, where we defined the Laplacian on the full shift space. Interested readers may refer \cite{arxiv} for complete details. In section \eqref{S decimation}, we construct an algorithm to extend an eigenvalue say $\lambda_{m - 1}$ of $H_{m - 1}$ with corresponding eigenfucntion $u_{m-1}$ to an eigenvalue $\lambda_{m}$ of $H_{m}$ with corresponding eigenfunction $u_{m}$. This method is due to Rammal \cite{rammal} and is termed as spectral decimation method. In section \eqref{cts ext}, we prove that such sequence of eigenfunctions can be continuously extended to the entire space $\Sigma_{N}^{+}$. Sections \eqref{forbidden ev} and \eqref{D spectrum} are devoted to determine the complete spectrum of the difference operators $H_{m}$ under the Dirichlet boundary conditions. In section \eqref{main theorem}, we state and prove the main theorem of this paper, proving the existence of a Dirichlet eigenvalue of the Laplacian on $\Sigma_{N}^{+}$.

\section{Preliminaries}
\label{preliminaries}
In this section, we outline the basic settings of the shift space necessary to define the Laplacian. The reader may consult \cite{arxiv} for complete details. The one-sided symbolic space on $N>1$ symbols, say $S := \{ 1, 2, \cdots, N \}$, is defined as,
\[ \Sigma_{N}^{+}\ \ :=\ \ S^{\mathbb{N}}\ \ =\ \ \{ x = (x_{1}\, x_{2}\, \cdots)\ :\ x_{i} \in S \}. \]
This space is equipped with the usual left shift operator $\sigma : \Sigma_{N}^{+} \longrightarrow \Sigma_{N}^{+}$, given by, $\sigma (x_{1}\, x_{2}\, \cdots) := (x_{2}\, x_{3}\, \cdots)$. The pair $(\Sigma_{N}^{+}, \sigma)$ is called the \emph{full one-sided shift on $N$ symbols}. The standard metric on $\Sigma_{N}^{+}$ is defined as,
\[ d(x, y)\ \ :=\ \ \frac{1}{2^{\, \rho(x, y)}},\ \ \ \text{where}\ \ \rho(x, y)\ :=\ \min \{ i : x_{i} \ne y_{i} \}\ \ \text{with}\ \rho(x, x)\ :=\ \infty. \] 
The symbol set $S$ is endowed with a discrete metric and thus, the space $\Sigma_{N}^{+}$ has a product topology, which is generated by the metric $d$ defined above. A \emph{cylinder set of length $m$}, $m \ge 1$, is a set given by, 
\[ [p_{1}\, \cdots\, p_{m}]\ \ :=\ \ \left\lbrace x \in \Sigma_{N}^{+} : x_{1} = p_{1}\, ,\, \cdots\, ,\, x_{m} = p_{m} \right\rbrace, \]
where the initial $m$ letters of a word are fixed. In the product topology of $\Sigma_{N}^{+}$, the cylinder sets are both closed and open, and they form a basis for the topology. Observe that the cylinder sets can be positioned anywhere in a word, but for simplicity, we fix them at the initial co-ordinates. It is an easy observation that the space $\Sigma_{N}^{+}$ is totally disconnected, compact metric space. The shift map $\sigma$ is a non-invertible, continuous surjection that has $N$ local inverse branches. 
\medskip 

\noindent
The equidistributed Bernoulli measure $\mu$ on $\Sigma_{N}^{+}$ is  determined by its restriction to the semi-algebra consisting of the cylinder sets, which in turn generate the Borel sigma algebra on $\Sigma_{N}^{+} $, as, 
\begin{equation} 
\label{measure} 
\mu \left( [p_{1}\, \cdots\, p_{m}] \right)\ \ :=\ \ \frac{1}{N^{m}}. 
\end{equation}   
For any $m \ge 1$, consider a finite word $w = (w_{1}\, w_{2}\, \cdots\, w_{m})$ of length $|w| = m$. Define a map $\sigma_{w} := \sigma_{w_{1}} \circ \sigma_{w_{2}} \circ \cdots \circ \sigma_{w_{m}} : \Sigma_{N}^{+} \longrightarrow [w_{1}\, w_{2}\, \cdots\, w_{m}]$ that concatenates the word $w$ as a prefix to the words in $\Sigma_{N}^{+}$. We then have a self-similar structure on the shift space as,
\[ \Sigma_{N}^{+} = \bigcup\limits_{ \left\lbrace  w\,:\, |w|\, =\, m \right\rbrace } \sigma_{w} (\Sigma_{N}^{+}). \]
Let $V_{0}$ be the set of all fixed points of $\sigma$, namely,
\[ V_{0} \ := \ \left\lbrace (\dot{1})\, ,\, (\dot{2})\, ,\, \cdots\, ,\, (\dot{N})\, \right\rbrace,\ \text{where}\ (\dot{l})\ =\ (l\, l\, l\, \cdots\, )\ \text{for}\ l \in S. \] 
Self-similarity helps us to construct an increasing sequence $\{  V_{m}\}_{m \ge 1}$ of subsets of $\Sigma_{N}^{+}$ as,
\[ V_{m}\ \ := \ \  \bigcup\limits_ {\left\lbrace  w\,:\, |w|\, =\, m \right\rbrace} \sigma_{w}\, (V_{0}), \]
such that the set $V_{*} := \bigcup\limits_{m\, \ge\, 0} V_{m} $ is dense in $\Sigma_{N}^{+}$. Any point $p=(p_{1}\, \cdots\, p_{m}\, p_{m + 1}\, p_{m+1}\, \cdots \, ) \in V_{m}$ is denoted by $(p_{1}\,  \cdots\, p_{m}\, \dot{p}_{m + 1})$. In particular, if $p \in V_{m} \setminus V_{m - 1}$, we have  $p_{m} \ne p_{m + 1}$. 
\medskip

\noindent
Let $m \ge 0$. Any two  points $p = (p_{1}\,  \cdots\, p_{m}\, \dot{p}_{m + 1} )$ and $ q = (q_{1}\,  \cdots\, q_{m}\, \dot{q}_{m + 1})$  in $V_{m}$ are  said to be \emph{$m$-related}, denoted by $p \sim_{m} q$, if $p_{i} = q_{i}$ for $1 \le i \le m$. This is an equivalence relation on $V_{m}$. The equivalence class of a point $p$ contains exactly $N$ points of $V_{m}$ and is given by, 
\[ \left. [p_{1}\,\cdots\, p_{m}] \right|_{V_{m}}\ \ :=\ \ \left\lbrace (p_{1}\, p_{2}\, \cdots\, p_{m}\,\dot{l})\ :\ l \in S \right\rbrace\ \ =\ \ [p_{1}\,\cdots\, p_{m}]\, \cap\, V_{m}. \] 
The \emph{deleted neighbourhood of $p$ in $V_{m}$}, denoted by $\mathcal{U}_{p,\, m}$, is defined as,
\[ \mathcal{U}_{p,\, m} \ \ := \ \ \left. [p_{1}\,\cdots\, p_{m}] \right|_{V_{m}} \setminus \left\lbrace p \right\rbrace \]  
which consists of the $N-1$ points that are $m$-related to $p$, other than $p$. These points are called as the \emph{immediate neighbours} of $p$ in $V_{m}$, since they are the closest to $p$ at a distance $\frac{1}{2^{m + 1}}$ in $V_{m}$. We fix the notation $q^{1},\, q^{2},\,\cdots, \, q^{N-1} $, for the immediate neighbours of $p$. In particular, if $p \in V_{m} \setminus V_{m-1}$, then one of these immediate neighbours of $p$, say $q^{N-1} := (p_{1}\, p_{2}\, \cdots\, p_{m - 1}\, \dot{p}_{m}) $ belongs to $V_{m-1}$ and the remaining $N-2$ neighbours form the set 
\[ U_{p,\, m}\ \ :=\ \ \left\lbrace  (p_{1}\, p_{2}\, \cdots\, p_{m}\, \dot{l}\, )\ :\ l \ne p_{m} \ \text{and}\ l \ne p_{m + 1} \right\rbrace \ = \ \left\lbrace q^{1},\, q^{2},\,\cdots, \, q^{N-2} \right\rbrace \ \subset\ \ V_{m} \setminus V_{m - 1}. \]
\medskip

\noindent
Let $ (\dot{l}) \in V_{0}$ and $p \in V_{m} \setminus V_{m-1}$. Choose $n_{1},\, n_{2},\, \cdots\, ,n_{d} \in \mathbb{N}$ such that $1 \le n_{1} < n_{2} < \cdots < n_{d} = m$, which are the only coordinates of $p$ satisfying $p_{n_{i}} \ne p_{n_{i}+1}$. Consider the points
\[r^{0} = (\dot{p}_{1}) \in V_{0};\ \ r^{n_{i}} = (p_{1}\, p_{2}\, \cdots\, p_{n_{i}}\, \dot{p}_{n_{i}+1}) \in V_{n_{i}} \setminus V_{n_{i} - 1}\ \ \text{and}\ \ r^{n_{d}} = p.   \]
Note that these points form a chain connecting $(\dot{l}) $ and $p$ with the relation
\[ (\dot{l}) \sim_{0 } r^{0} \ \ \ \ \text{and}\ \ \ \ r^{n_{i - 1}} \sim_{n_{i}} r^{n_{i}}\ \ \text{implying}\ \ r^{n_{i - 1}} \in \mathcal{U}_{r^{n_{i}},\, n_{i}}\ \ \ \text{for}\ 1 \le i \le d. \] 
Even though the space $\Sigma_{N}^{+}$ is totally disconnected, owing to the arguments in \cite{arxiv}, the set $V_{0}$ is defined as the \emph{boundary} of $\Sigma_{N}^{+}$. For $m \ge 0$, the set of all real valued functions on $V_m$ is denoted by $\ell(V_m)$. The standard inner product on $\ell(V_m)$ is given by, $\left\langle u,v \right\rangle = \sum\limits_{p \in V_m} u(p) v(p)$. 
\medskip

\noindent
The \emph{Laplacian on $\Sigma_{N}^{+}$} is now defined as a renormalized limit of the difference operators $H_{m}$ on $\ell(V_{m})$, see \cite{arxiv}. The matrix representation of these operators is easier to deal with, and takes into consideration the ordering of points in $V_{m}$. Whenever a point $p$ appears before $q$ in the ordering of $V_{m}$, we denote it by $p \prec q$ and in that case, a row (column) corresponding to a point $p$ appears to the left (top) to a row (column) corresponding to a point $q$, in the matrix representation of $H_{m}$. The set $V_{0}$ can be written in an ascending order of its elements as, 
\[ V_{0} \ \ = \ \ \left\{ (\dot{1}) \prec (\dot{2}) \prec \cdots \prec (\dot{N}) \right\}.  \]
That is, for $k,\, l \in S,\ (\dot{k}) \prec (\dot{l})$ if and only if $k < l$. Considering the symbol set $S$ as a subset of the natural numbers $\mathbb{N}$, the ordering on $S$ is assumed to be the natural ordering on $\mathbb{N}$. A difference operator $H_{0}$ on $V_{0}$ is given by,
\begin{equation*}
\left( H_{0} \right)_{pq}\ \ :=\ \ 
\begin{cases} 
1 & \text{if}\ q \in \mathcal{U}_{p,\, 0} \\
-\, (N - 1) & \text{if}\ p = q,
\end{cases} 
\end{equation*} 
where $\left( H_{0} \right)_{pq}$ denotes the $(p, q)$-th entry of the matrix $H_{0}$ corresponding to the row for $p \in V_{0}$ and the column for $q \in V_{0}$. The action of $H_{0}$ on any $u \in \ell(V_{0})$ is explicitly described as,
\begin{equation} 
\label{defn H_0} 
(H_{0} (u))\, (\dot{l})\ \ =\ \  \sum_{q\, \in\, V_{0}} \left( H_{0} \right)_{pq}\, u(q) \ \ =\ \ -\, (N - 1)\, u(\dot{l}) + \sum\limits_{\substack{k\, \in\, S \\ k\, \ne\, l}} u(\dot{k}). 
\end{equation}
Let $m \ge 1$. For $l \in S$, denote by $\sigma_{l} : \Sigma_{N}^{+} \longrightarrow \left[ l\right] $ the map $\sigma_{l}(x):= (l\,x_{1}\,x_{2}\cdots)$. Then $V_{m} = \bigcup\limits_{l\, \in\, S} \sigma_{l}(V_{m-1}) $. Proceeding inductively, we define an order on the elements of $V_{m}$. Here $V_{m-1}$ appears first in the ordering of $V_{m}$ retaining the order therein. Now, for any $p,\,q \in V_{m-1}$ and $i,\,j \in S$, we have $\sigma_{i}(p) \prec \sigma_{j}(p) $ if and only if $i < j$ and $ \sigma_{i}(p) \prec \sigma_{j}(q)$ if and only if $ p \prec q$. Thus, the elements in $V_{m}$ can be listed in their ascending order as,  
\begin{eqnarray*} 
V_{m} & = & \Bigg\{\ \underbrace{(\dot{1}) \prec \cdots \prec (\dot{N})}_{V_{0}} \prec \underbrace{(2 \dot{1}) \prec \cdots \prec (N \dot{1}) \prec \cdots \prec (1 \dot{N}) \prec \cdots \prec (N - 1\, \dot{N})}_{V_{1} \setminus V_{0}} \prec  \\ 
& & \\
& & \hspace{+2cm} \underbrace{\cdots \prec \cdots \prec \cdots \prec \cdots \prec}_{V_{2} \setminus V_{1},\, \cdots,\, V_{m - 1} \setminus V_{m - 2}} \\
& & \hspace{+5cm} \underbrace{(\underbrace{1\, \cdots\, 1}_{m - 1}\, 2\, \dot{1}) \prec \cdots \prec (\underbrace{N\, \cdots\, N}_{m - 1}\, N - 1\, \dot{N})}_{V_{m} \setminus V_{m - 1}}\ \Bigg\}. 
\end{eqnarray*} 

\noindent
An appropriate difference operator $H_{m}$ on $\ell (V_{m})$ is then defined as follows:

\begin{eqnarray*}
(H_{m})_{pq}\ \ & = &\ \ 
\begin{cases} 
-(m + 1)\,(N-1) & \ \ \text{if}\ p = q \in V_{0},\\
-(m-n+1)\,(N-1) & \ \ \text{if}\ p = q \in V_{n} \setminus V_{n-1} \ \text{for} \  0< n \le m,\\
1 & \ \ \text{if} \ p \in V_{0} \ \text{and} \ q \in \mathcal{U}_{p,\, i}\ \text{for} \ 0 \le i \le m,\ \text{or},\\
& \ \ \ \ \ p \in V_{n} \setminus V_{n-1} \ \text{and} \ q \in \mathcal{U}_{p,\, i}\ \text{for} \ n \le i \le m, \\
0 & \ \ \text{otherwise}. \\ 
\end{cases}  \nonumber\\ 
\end{eqnarray*}

\noindent
Since $H_{m}$ acts on any $u \in \ell(V_{m})$ according to the equation $ H_{m} u (p)\ = \ \sum\limits_{q\, \in V_{m}}\, (H_{m})_{pq}\,u(q) $, this action can be described seperately on $V_{m} \setminus V_{m - 1}$ and $V_{m-1}$ as follows:
\medskip 

\noindent 
For $p \in V_{m} \setminus V_{m - 1}$ we have,
\begin{eqnarray} 
\label{defn_of_Hm_u}
H_{m} u\, (p) \ \ = \ \ 
\begin{cases}
-\, (N - 1)\, u(p)\, +\, \sum\limits_{q \, \in \,\mathcal{U}_{p, m}} u(q)  &\text{if} \ \ p \in V_{m} \setminus V_{m - 1}, \\
H_{m - 1} u\,(p)\, +\, \left[ -\, (N - 1)\, u(p)\, + \sum\limits_{q \, \in \, \mathcal{U}_{p, m}} u(q) \right] &\text{if} \ \ p \in V_{m-1}.
\end{cases}
\end{eqnarray}

\noindent
Corresponding to the equidistributed Bernoulli measure $\mu$ defined in equation \eqref{measure}, consider the set
\begin{eqnarray} 
\label{laplacian}
D_{\mu} & := & \Bigg\{ u \in \mathcal{C}(\Sigma_{N}^{+}) : \exists\, f \in \mathcal{C}(\Sigma_{N}^{+})\ \text{satisfying} \nonumber \\ 
& & \hspace{+2cm} \lim_{m\, \to\, \infty}\ \max_{p\, \in\, V_{m} \setminus V_{m - 1}} \left| \frac{H_{m} u (p)}{\mu([p_{1} p_{2} \cdots p_{m + 1}])}\, - f(p) \right| = 0 \Bigg\}. 
\end{eqnarray} 
Then such a function $f$ corresponding to the function $u \in D_{\mu} $ is called the \emph{Laplacian} of $u$ and is denoted by $f = \Delta_{\mu} u$. It follows directly from the definition, that the Laplacian can be expressed as a pointwise limit as,
\begin{equation} 
\label{pointwise laplacian}
f(x)\ \ =\ \ \Delta u \,(x)\ \ =\ \ \lim_{m\, \rightarrow\, \infty} N^{m + 1}\, H_{m} u (p^{m}).
\end{equation}
Here, $\left\lbrace p^{m} = (x_{1}\, x_{2}\, \cdots\, x_{m}\, \dot{l}) \, \in V_{m} \setminus V_{m-1} \right\rbrace_{m\, \ge\, 1} $ is a sequence of points converging to $x$.

\section{Spectral decimation}
\label{S decimation}
In the following sections, we aim to determine the Dirichlet spectrum of the Laplacian. The standard eigenvalue equation for $ \Delta$ is written as, 
\[ \Delta u = -\lambda u \ \ \text{on } \ \ \Sigma_{N}^{+} \setminus V_{0}. \]  
Here, $\lambda $ is known as the eigenvalue and $u$, the corresponding eigenfunction of $\Delta$. Since the Laplacian is the limit of the sequence of difference operators $H_{m}$ with the scaling factor ${N^{m+1}}$, the first step towards determining the spectrum of the Laplacian, is to study the structure of the eigenvalues of every $H_{m}$. The eigenvalue equation for $H_{m}$ looks like,
\[H_{m} u_{m}\ \ =\ \ - \lambda_{m} u_{m} \ \ \text{on } \ \ V_{m} \setminus V_{0}. \]
The function $u_{m} \in \ell(V_{m})$ is the eigenfunction of $H_{m}$ corresponding to the eigenvalue $\lambda_{m}$. The natural choice for the eigenvalue of $\Delta $ is then 
\[\lambda \ \ =\ \ \lim\limits_{m\, \rightarrow\, \infty} N^{m+1} \lambda_{m},\] whenever the limit exists. 
\medskip 

\noindent 
We begin with finding an extension of an eigenfunction of $H_{m-1}$ to an eigenfunction of $H_{m}$ and the relation between these corresponding eigenvalues. This method is called the \emph{spectral decimation} and is discussed for Sierpinski gasket in detail in \cite{fuku_shima}, \cite{shima}. We adapt the idea in the settings of symbolic space.
\medskip

\noindent
Let $u_{m-1} \in \ell(V_{m-1})$ be such that it is nonzero on the set $V_{m-1} \setminus V_{m-2}$ and $ u_{m} \in \ell(V_{m})$ be its extension to $V_m$. Let $ \lambda_{m-1},\, \lambda_{m} \in \mathbb{ R}$ be such that $\lambda_{m-1} \neq 0$ and $\lambda_{m} \neq 0,1,N $. We call these values as the \emph{forbidden eigenvalues}. It will become clear in the proof of the next lemma, as to why these values are named so. We deal with the cases when $\lambda_{m-1} = 0$ and $\lambda_{m} \,=\, 0,1 \text{ or }N $ in a later section.

\begin{lemma}\label{lem1}
For $m \ge 1$, any eigenfunction $u_{m}$ of $H_{m}$ with eigenvalue $\lambda_{m}$ satisfies, 
\begin{equation} \label{ext to V_m}
u_{m}(p) = \frac{u_{m}(q^{N-1})}{(1-\lambda_{m})},
\end{equation}
for any $p \in V_{m} \setminus V_{m-1}$ and  $q^{N-1} \in \mathcal{U}_{p,\,m}$, where the notation is as defined in section \eqref{preliminaries}.
\end{lemma}

\begin{proof}
If $u_{m}$ is an eigenfunction of $H_{m}$, it satisfies the eigenvalue equation $H_{m} u_{m} = -\lambda_{m} u_{m}$ on $V_{m} \setminus V_{0}$ and hence on $V_{m} \setminus V_{m-1} $. Let $p \in V_{m} \setminus V_{m-1}$ and $q^{1}, \, q^{2},\, \cdots,\, q^{N-1} \in \mathcal{U}_{p,\,m}$ as defined in section \eqref{preliminaries}. By the definition of $H_{m}$  as in equation \eqref{defn_of_Hm_u}, we write, 
\begin{equation*} 
H_{m} u_{m}(p)\ \  =\ \ -(N-1)\, u_{m}(p)\,+ \,u_{m}(q^{1})\, + \,u_{m}(q^{2})\, +\, \cdots \,+\,u_{m-1}(q^{N-1}) \ \ = \ \ -\lambda_{m} u_{m}(p) 
\end{equation*}
Similarly we write the eigenvalue equation for the points $q^1,  q^2, \cdots , q^{N-2} $ to obtain the following system:
\begin{equation*}
-(N-1) \,u_{m}(q^{i})\, +\, \left( u_{m}(p)\, +\, \sum\limits_{\substack{j\, =\, 1 \\ j\, \ne\, i}}^{j = N} u_{m}(q^{j}) \right) \ \ =\ \ -\lambda_{m} u_{m}(q^{i}) \quad \text{for } 1 \le i \le N-2.
\end{equation*}
Rearranging the terms in above two equations, we get the system:
\begin{align}
\sum\limits_{j\, =\, 1}^{N-1} u_{m}(q^{j}) \ \ &=\ \ \left[ (N-1)- \lambda_{m} \right]\, u_{m}(p), \label{eq1}  \\
u_{m}(p) + \sum\limits_{\substack{j\, =\, 1 \\ j\, \ne\, i}}^{j\, =\, N} u_{m}(q^{j}) \ \ &=\ \  \left[ (N-1)- \lambda_{m} \right]\, u_{m}(q^{i}) \qquad \text{for } 1 \le i \le N-2. \label{eq2}
\end{align}
Adding all the equations in the above system we get, 
\[ (N-2) \left[ u_{m}(p)\, + \sum\limits_{j\, =\, 1}^{N - 2} u_{m}(q^{j})  \right] \,+\, (N-1)\,u_{m}(q^{N-1}) \ \ =\ \ \left[ N - 1 - \lambda_{m} \right]\,   \left[ u_{m}(p)\, + \sum\limits_{j\, =\, 1}^{N - 2} u_{m}(q^{j})  \right].  \]
For $\lambda_{m} \ne 1$, a simple rearrangement of the terms gives the relation,
\begin{equation} \label{eq3}
u_{m}(p) + \sum\limits_{j\, =\, 1}^{N - 2} u_{m}(q^{j}) \ \ =\ \ \frac{u_{m}(q^{N-1})\, (N-1)}{1-\lambda_{m}}.
\end{equation} 
We now add $u_{m}(p)$ in both sides of equation (\ref{eq1}) and $u_{m}(q^{i})$ in both sides of the respective equation in (\ref{eq2}), to obtain,
\[ u_{m}(p) + \sum\limits_{j\, =\, 1}^{N - 1} u_{m}(q^{j})\ = \ \left[ N- \lambda_{m} \right] u_{m}(p) \ = \ \left[ N- \lambda_{m} \right] u_{m}(q^{1})\ = \ \cdots \ = \ \left[ N- \lambda_{m} \right] u_{m}(q^{N-2}). \]
Assuming $ \lambda_{m} \ne N$, we get,
\begin{equation*}
u_{m}(p) \ \ = \ \ u_{m}(q^{i}) \quad \text{for } 1 \le i \le N-2 .
\end{equation*}
Substituting this condition in equation (\ref{eq3}), we get the required relation, 
\begin{equation}\label{ext_efunction}
u_{m}(p) \ \ = \ \ u_{m}(q^{i})\ \ = \ \ \frac{u_{m - 1}(q^{N-1})}{(1-\lambda_{m})} \qquad \text{for } 1 \le i \le N-2,
\end{equation}
where $u_{m - 1} \equiv u_{m}|_{V_{m - 1}} \in \ell(V_{m - 1})$.
\end{proof}
\medskip

\noindent
Now suppose that $u_{m-1}$ is an eigenfunction of $H_{m-1}$ with eigenvalue $ \lambda_{m-1}$ and some extension of $u_{m-1}$, say $u_{m}$, is an eigenfunction of $H_{m}$ with eigenvalue $ \lambda_{m}$. We are interested in finding the relationship between the two eigenvalues $\lambda_{m-1} $ and $\lambda_{m}$. Choose a point $q \in V_{m-1} \setminus V_{m-2}$ such that $u_{m-1}(q) \ne 0$. Since $q \in V_{m}$, $u_{m-1}(q)\,=\, u_{m}(q)$. We write the two eigenvalue equations for this point $q$ as,
\begin{align}
H_{m-1} u_{m-1}(q)\ \ &=\ \ -\lambda_{m-1}\,u_{m-1}(q); \nonumber\\
H_{m} u_{m}(q)\ \ &=\ \ -\lambda_{m}\,u_{m}(q). \label{evequation}
\end{align} 
Applying the definition of $H_m$ on $V_{m-1}$ as in equation \eqref{defn_of_Hm_u} and using the relation \eqref{ext_efunction} from the previous lemma we obtain,
\begin{align*}
H_{m} u_{m}(q)\ \ &=\ \ H_{m-1} u_{m-1}\,(q)\,+\, \Bigg[ -(N-1)\,u_{m-1}(q)\,+  \sum\limits_{s \in \mathcal{U}_{q,\,m} } u_{m}(s)  \Bigg]\\
&=\ \ -\lambda_{m-1}u_{m}\,(q) \,+\, \Bigg[ -(N-1)\,u_{m}(q)\,+\,(N-1)\, \frac{u_{m}(q)}{1 - \lambda_m }   \Bigg]. \\
\end{align*}
Substituting this expression for $ H_m u_m\,(q)$ in equation (\ref{evequation}), we get,
\begin{align*}
 -\lambda_{m-1}-(N-1)\, +\, \frac{(N-1)}{1-\lambda_{m}}\ \ =\ \ -\lambda_{m}, \quad \text{as }\ \ u_{m}(q) \ne 0.
\end{align*}
A simple rearrangement yields the relation,
\begin{equation}\label{lambda}
\lambda_{m-1} \ \ =\ \ \frac{\lambda_{m}(N-\lambda_{m})}{(1-\lambda_{m})}.
\end{equation} 
Moreover, solving this as a quadratic equation for $\lambda_{m}$, we get an equivalent relation expressing $\lambda_{m} $ in terms of $\lambda_{m-1}$ as,
\begin{equation}\label{lambda2}
\lambda_{m}\ \ =\ \ \frac{(N+\lambda_{m-1}) + \beta_{m} \sqrt{(N+\lambda_{m-1})^2- 4\lambda_{m-1}} }{2}, \quad \text{where } \beta_{m} = \pm 1.
\end{equation}
\medskip 

\noindent 
 \begin{theorem}[Spectral Decimation method]
\label{decimation}
Let $m \ge 2$. Suppose $\lambda_{m - 1}$ is an eigenvalue of $H_{m - 1}$ with corresponding eigenfunction $u_{m-1}$. Define $\lambda_{m}$ in terms of $\lambda_{m-1} $ according to equation \eqref{lambda} (or equivalently \eqref{lambda2}). Let $u_{m}$ be the extension of $u_{m - 1}$ be given by equation \eqref{ext_efunction}. Then $u_m$ satisfies the eigenvalue equation 
\[H_{m}\,u_{m}\ \ =\ \ -\lambda_{m}\,u_{m} \quad \text{on } V_{m}\setminus V_{0}.\]
Conversely, suppose $u_{m}$ is a $\lambda_{m}$-eigenfunction of $H_{m}$. Let $\lambda_{m-1} \in \mathbb{R}\setminus \{0\}$ be related to $\lambda_{m}$ according to equation (\ref{lambda}). Then, the restriction $u_{m-1} $ on $V_{m-1}$ is an eigenfunction of $H_{m-1}$ with eigenvalue $\lambda_{m-1}$.
\end{theorem}
\begin{proof}
To prove the first part of the theorem, we assume that $u_{m-1}$ satisfies the eigenvalue equation $H_{m-1}\,u_{m-1}\, =\, -\lambda_{m-1}\,u_{m-1} $ on $V_{m-1} \setminus V_{0}$. Suppose that $u_{m-1}$ is extended to a function $u_{m} \in \ell(V_{m})$ according to the equation \eqref{ext_efunction}. For $\lambda_{m}$ related to $\lambda_{m-1}$ as given by equation \eqref{lambda}, we prove that $u_m$ satisfies the eigenvalue equation $H_m\,u_m\, =\, -\lambda_m\,u_m $ on two sets $V_m \setminus V_{m-1}$ and $V_{m-1} \setminus V_{0}$.
\medskip

\noindent
Let $p \in V_{m} \setminus V_{m-1}$ and $q^{1},  q^{2}, \cdots , q^{N-1} \in \mathcal{U}_{p,\,m} $ be as before. Since these points satisfy the equation \eqref{ext_efunction}, from the definition of $ H_{m} u_{m}(p)$ as in equation \eqref{defn_of_Hm_u}, we obtain,
\begin{align*}
H_{m}\,u_{m}(p) \ \ &=\ \ -(N-1)\,u_{m}(p)\, +\, u_{m}(q^{1})\, +\, u_{m}(q^{2})\,+\, \cdots\, + \,u_{m}(q^{N-2})\,+\,u_{m-1}(q^{N-1}) \\
&=\ \ -(N-1)\, \frac{u_{m-1}(q^{N-1})}{ 1-\lambda_{m}} \, + \, (N-2)\,\frac{u_{m-1}(q^{N-1}) }{ 1-\lambda_{m}} \,+\,u_{m-1}(q^{N-1}) \\
&=\ \ -\frac{u_{m-1}(q^{N-1}) }{ 1-\lambda_{m}} \,+\,u_{m-1}(q^{N-1}) \\
&=\ \ u_{m}(p)\, \big[ -1\,+\,(1-\lambda_{m}) \big]\\
&=\ \ -\lambda_{m} \,u_{m}(p). 
\end{align*}
Now for $p \in V_{m-1}$, recall that there are $N-1$ points in $V_{m} \setminus V_{m-1} $ which are $m$-related to $p$ collected in the set $\mathcal{U}_{p,\,m}$. Thus, by equation \eqref{defn_of_Hm_u} we have,
\begin{align*}
H_{m}\,u_{m}\,(p)\ &=\ H_{m-1} u_{m-1}\,(p)\,+\, \Bigg[ -(N-1)\,u_{m-1}(p)\,+  \sum\limits_{q \in \mathcal{U}_{p,\,m}} u_{m}(q)  \Bigg]\\
&=\ -\lambda_{m-1}  u_{m}\,(p)\,+\,\Bigg[ -(N-1)\,u_{m}(p)\,+\, (N-1)\,\frac{u_{m}(p)}{1-\lambda_{m} }  \Bigg]\\
&=\  \Big[(-\lambda_{m-1}-N+1)\,(1-\lambda_{m} )\,+\,N-1  \Big]\, \frac{u_{m}(p)}{1-\lambda_{m} }.
\end{align*}
After substituting for $\lambda_{m-1}$ from equation (\ref{lambda}) and rearranging, we obtain,
\begin{align*}
H_{m}\,u_{m}\,(p)\ \ &=\ \ \frac{\Big[ -\lambda_{m} (N-\lambda_{m} )\,-\,(N-1)\,(1-\lambda_{m} )  \,+\,N-1  \Big]}{1-\lambda_{m}}\,u_{m}(p)\\
&=\ -\lambda_{m}\, u_{m}(p)
\end{align*}
Therefore $u_{m}$ is an eigenfunction of $H_{m}$ with eigenvalue $\lambda_{m}$.
\medskip

\noindent
The proof of the converse follows directly by reversing the arguments above. However, we include only a short proof for the readers' convenience. Suppose $u_{m}$ satisfies the eigenvalue equation $H_{m}\,u_{m}\, =\, -\lambda_{m}\,u_{m}$ on $V_{m}$. Then for any $q \in V_{m-1}$,
\[H_{m}\,u_{m}\,(q) \ = \ H_{m-1}\,u_{m-1}(q)\,+\,  \Bigg[ -(N-1)\,u_{m-1}(q)\,+  \sum\limits_{s\, \in\, \mathcal{U}_{q,\,m}} u_m(s)  \Bigg]\ = \ -\lambda_m\,u_m\,(q).  \] 
This implies,
\begin{align*}
H_{m-1}\,u_{m-1}(q)\ &= \ \Bigg[-\lambda_m\,+\,N\,-\,1\,-\,\frac{N-1}{1-\lambda_m}  \Bigg] \,u_{m-1}(q) \\
&=\ - \frac{\lambda_{m}(N-\lambda_{m})}{(1-\lambda_{m})} \,u_{m-1}(q)\\
&=\ -\lambda_{m-1}\,u_{m-1}(q).
\end{align*}
\end{proof}

\medskip

\section{Continuous extension of the eigenfunctions}
\label{cts ext}
Define a map $\phi_{\beta } : \mathbb{R} \longrightarrow \mathbb{R}$ as,
\begin{equation}
\label{phi}
\phi_{\beta }(x)= \frac{N+x}{2} \left[ 1 + \, \beta \sqrt{1- \frac{4x}{(N+x)^{2}}} \  \right], \ \ \text{for}\ \ \beta = \pm 1.
\end{equation}
For simplicity, we denote $\phi_{\beta}$ by $\phi_{+}$ or $\phi_{-}$, according to the value of $\beta = 1$ or $\beta = -1$ respectively. If we choose $\beta = -1$, then one can observe that $0$ is the only attracting fixed point of $\phi_{\beta}$. The eigenvalues of $H_{m-1}$ and $H_{m}$ that are related according to equation \eqref{lambda2} in the spectral decimation, satisfy $\lambda_{m} = \phi_{\beta_{m} } (\lambda_{m-1}) $, with $\beta_{m} = \pm 1$. In this paper, we study only the positive spectrum of the difference operators. For any $x > 0$, note that $\frac{4x}{(N+x)^{2}} < 1$, as $N>2$ and thus, all the eigenvalues  are real. Using the binomial expansion for $\sqrt{1-x}$, for $x \in (0, \, 1] $, we obtain,
\begin{equation}
\label{order}
\phi_{-}(x)\ \ = \ \ \frac{x}{N+x} \ +\, O\left(\frac{1}{N^{2}}\right).
\end{equation} 
For any $\lambda_{m} > 0$, if the value of $\beta_{m+1}$ is chosen to be $-1$, then we observe that $\lambda_{m+1} < 1$.  For, if not, then writing $\lambda_{m+1}$ according to the equation \eqref{lambda2}, we have,
\[ \lambda_{m+1}\ \ =\ \ \frac{(N+\lambda_{m}) - \sqrt{(N+\lambda_{m})^2- 4\lambda_{m}} }{2} \ \ > \ \ 1,\]
which implies,
\[ (N+\lambda_{1})\,-\,2 \ \ > \ \ \sqrt{(N+\lambda_{1})^2- 4\lambda_{1}}. \]
Squaring both the sides, we get,
\[ (N+\lambda_{m})^{2} - 4(N+\lambda_{m}) + 4 \ \ > \ \ (N+\lambda_{m})^{2} - 4 \lambda_{m}.\]
Simplifying, we arrive at $N < 1$, which is a contadiction.
\medskip

\noindent
Choose $\beta_{m} = - 1$ for all but finitely many $m \ge 1$. Throughout this paper, we denote the $m$-fold composition of $\phi_{-}$ with itself by,
\[ \phi_{-}^{m} \ \ := \ \ \underbrace{\phi_{-} \circ \cdots \circ \phi_{-} }_{m}   \]
 If $m_{0} := \min \left\lbrace m\, :\, \beta_{i} = -1, \ \forall \ i > m \right\rbrace  $, then $\lambda_{m} = \phi_{-}^{m-m_{0}} (\lambda_{m_{0}}) $ for all $m > m_{0}$. Observe that $\lambda_{m} \in (0,1)$ for all $m > m_{0}$ and the sequence $\{ \lambda_{m}\}_{m\, \ge\, m_{0}}$ is decreasing, with the order of convergence given by equation \eqref{order}. Therefore, we conclude that the limit $\lim\limits_{m\, \to\, \infty} N^{m+1} \lambda_{m}$ exists. 
\medskip

\noindent
Let us begin with an eigenfunction $u_{\eta}$ of $H_{\eta}$ with the corresponding eigenvalue $\lambda_{\eta}$, for some $\eta \ge 1$. The method of spectral decimation provides us a sequence of eigenfunctions $\{ u_{m}\}_{m\, \ge\, \eta}$ and corresponding eigenvalues $\{\lambda_{m} \}_{m\, \ge\, \eta}$. For $m < \eta $, denote the restriction of $u_{\eta}$ to $V_{m}$ by $u_{m}$. Define a function $u : \Sigma_{N}^{+} \longrightarrow \mathbb{R}$ as follows: Let $x \in \Sigma_{N}^{+}$. If $x \in V_{*}$, define $u(x) := u_{m}(x)$, where $m = \min \{ i \ge 0 \, : \, x \in V_{i}\} $. For $x \in \Sigma_{N}^{+} \setminus V_{*}$, consider a sequence of points $\{ p^{m_{k}} \}_{k\, \ge\, 0}$ converging to $x$, such that $p^{m_{k}} \in V_{m_{k}} \setminus V_{m_{k} - 1} $. Then define,
\begin{equation}\label{ext full}
u(x) := \lim\limits_{k\, \rightarrow\, \infty} u_{m_{k}} (p^{m_{k}}).
\end{equation} 
Note that this extension is independent of the choice of sequence in $V_{*}$ converging to $x$. 

\medskip

\noindent
Now, for any $q \in V_{\eta}$, we may consider all the points of $V_{*} \setminus V_{\eta}$ that are $\eta$-related to $q$, to be lying on a fibre, with $q$ being the base point of the fibre. If $q \in V_{\eta} \setminus V_{\eta-1}$, then the fibre corresponding to any point $p \in U_{q,\,\eta}$ coincides with the fibre corresponding to $q$. The value of the function $u$ along any of these fibres is determined by the value of $u_{\eta}$ at the base point and the sequence of eigenvalues $\{\lambda_{m} \}_{m\, \ge\, \eta}$.
\medskip

\noindent
In particular, for any $n_{0} > \eta$, choose a point $p = (p_{1}\,\cdots p_{\eta}\, \cdots\, p_{n_{0}}\,\dot{p}_{n_{0}+1})$ in $V_{n_{0}} \setminus V_{n_{0}-1}$, with $p_{n_{0}} \ne p_{n_{0}+1} $. The base point corresponding to $p$ is $ q = (p_{1}\, p_{2} \, \cdots \,p_{\eta}\, \dot{p}_{\eta+1}) \in V_{\eta}$. Let $\eta < n_{1} < n_{2} < \cdots < n_{d} = n_{0}$ be the only positions of letters (after $\eta^{th}$ position) in $p$ where for each $1 \le i \le d$, the consecutive letters $p_{n_{i}}$ and $p_{n_{i}+1}$ are distinct. 
\medskip

\noindent
Applying lemma \eqref{lem1} iteratively to each of the eigenfunctions $u_{n_{i}}$ for $1 \le i \le d$, we obtain the value of $u_{n_{0}}$ at the point $p$ as,
\begin{equation}
\label{eq4} 
u_{n_{0}}(p) \ \ = \ \ \frac{u(q)}{(1-\lambda_{n_{1}}) (1-\lambda_{n_{2}}) \cdots (1-\lambda_{n_{d}}) }.
\end{equation}
In particular, if all the consecutive letters from $\eta^{th}$ position to $(n_{0}+1)^{st}$ position of the point $p \in V_{n_{0}} \setminus V_{n_{0}-1}  $ differ from each other, then all the eigenvalues from $\lambda_{\eta + 1}$ to $\lambda_{n_{0}}$ contribute in determining the value of $u_{n_{0}}$ at $p$ as,
\[ u_{n_{0}}(p)\ \ = \ \ \frac{u(q)}{(1-\lambda_{\eta + 1}) (1-\lambda_{\eta + 2}) \cdots (1-\lambda_{n_0}) }. \]

\medskip

\noindent
For simplicity, let us assume that $\beta_{m} = -1$ for all $m > \eta$. Clearly, $\lambda_{m} \in (0,\,1)$ and hence $\frac{1}{1-\lambda_{m}} > 1$ for all $m > \eta$. Consider the sequence $\{a_{n} \}_{n\, >\, \eta}$ defined by,
\[ a_{n} \ \ := \ \ \frac{1}{(1-\lambda_{\eta + 1}) (1-\lambda_{\eta + 2}) \cdots (1-\lambda_{n})}. \] 
Convergence of this sequence $\{a_{n} \}_{n\, >\, \eta} $ guarantees the existence of the limit in (\ref{ext full}). Since $\Big|\frac{a_{n+1}}{a_{n}}\Big| \ = \ \frac{1}{(1-\lambda_{n+1})} \ > \ 1$, the sequence $\{a_{n} \}_{n\, >\, \eta} $ is increasing. Recall that, due to our particular choice of $\beta_{m}$, the limit $\lim\limits_{m\, \to\, \infty} N^{m+1} \lambda_{m}$ exists. Then there exist a positive real number $C $ and a natural number $M_{0} $ such that $\lambda_{n} \le \frac{C}{N^{n}}$ for all $n \ge M_{0}$. Therefore, we construct another increasing sequence $\{ b_{n}\}_{n\, >\, \eta}$ as,
\begin{equation*}
b_{n} := \frac{1}{(1-\lambda_{\eta + 1}) (1-\lambda_{\eta + 2}) \cdots (1-\lambda_{M_{0}})}\  \left[ \frac{N^{M_{0}+1} N^{M_{0}+2} \cdots N^{n}} {(N^{M_{0}+1}-C) (N^{M_{0}+2}-C) \cdots (N^{n}-C) }  \right] 
\end{equation*}
such that $|\,a_{n}\,| \le |\,b_{n}\,| $ for all $n \ge M_{0}$. Observe that this new sequence  $\{ b_{n}\}_{n > \eta}$ is contractive as, 
\begin{equation*}
\Bigg| \frac{b_{n+1}- b_{n}}{b_{n}- b_{n-1}}   \Bigg| =  \Bigg|\frac{N^{n}}{N^{n+1} - C}  \Bigg| < 1, \ \ \text{ whenever }\ \ n > \max \left\lbrace M_{0}, \frac{C}{(N-1) \log N} \right\rbrace
\end{equation*}
and hence converges to some limit, say $b$. Finally, the sequence $\{ a_n\}$ is also bounded above by $b$ and hence convergent.
\medskip

\noindent
Consider the sequence $c_{n_{0}} \ := \ \left\lbrace  \frac{1}{(1-\lambda_{n_{1}}) (1-\lambda_{n_{2}}) \cdots (1-\lambda_{n_{0}})} \right\rbrace_{n_{0}\, >\, \eta} $, obtained from the equation \eqref{eq4}. Since $\frac{1}{1-\lambda_{m}} > 1$ for all $m > \eta$, we obtain that $\{ c_{n}\}_{n\, >\, \eta}$ is an increasing sequence with $c_{n} < a_{n}$ for all $m > \eta$. Therefore $ \{ c_{n}\}_{n\, >\, \eta}$ converges and we conclude that any eigenfunction $u_{\eta}$ of $H_{\eta}$ can be extended to $\Sigma_{N}^{+}$ according to equation \eqref{ext full}. 
\medskip

\noindent
\begin{theorem}
The extension of an eigenfunction $u_{\eta}$ of $H_{\eta}$ with the corresponding eigenvalue $\lambda_{\eta}$ given by equation \eqref{ext full} is continuous on $\Sigma_{N}^{+} \setminus V_{*}$. 
\end{theorem} 
\begin{proof}
Let $u$ be the extension of a given eigenfunction $u_{\eta}$ to the entire space $\Sigma_{N}^{+}$ as given by equation (\ref{ext full}).
Let us take two points $x, y \in \Sigma_{N}^{+} \setminus V_{*}$ such that they agree on a large number of initial places, say $M_{0} > \eta$. We represent these points by $x = (x_{1}\, \cdots\, x_{\eta}\,\cdots\, x_{M_{0}} \,x_{M_{0}+1}\, \cdots )$ and $ y = (x_{1}\,\cdots \, x_{\eta} \,\cdots \, x_{M_{0}} \,y_{M_{0}+1}\,y_{M_{0}+2} \,\cdots )$. Clearly $x$ and $y$ belong to the same fibre starting at a base point $q = (x_{1}\, \cdots \, x_{\eta}\, \dot{x}_{\eta+1} ) \in V_{\eta}$. 
\medskip 

\noindent 
Construct an increasing sequence $n_{1}, n_{2} , \cdots \in \mathbb{N}$ as follows: Let $n_{1} > \eta $ be the first instance after $\eta$ such that $x_{n_{1}} \ne x_{n_{1} + 1}$. Similarly $n_{2} > n_{1}$ to be the first instance afrer $n_{1}$ for which $x_{n_{2}} \ne x_{n_{2}+1}$ and so on. This sequence includes all the positions of letters in $x$ which are different from their successive letter. Such a sequence exists because of the choice of $x \notin V_{*}$. Recall that $V_{*}$ consists of all the eventually constant sequences. Corresponding to each of these $n_{i}$, consider a point $p^{n_{i}} = (x_{1}\, x_{2}\, \cdots\, x_{n_{i}}\, \dot{x}_{n_{i}+1}) \in V_{n_{i}} \setminus V_{n_{i}-1} $ which naturally lies on the same fibre starting at $q$. By following the extension algorithm of the decimation method, extend $u_{\eta}$ to $V_{n_{i}} \setminus V_{n_{i}-1} $ as,
\begin{equation*}
u(p^{n_{i}}) \ \ = \ \ u_{n_{i}}(p^{n_{i}}) \ \ = \ \ \frac{u(q)}{(1-\lambda_{n_{1}}) (1-\lambda_{n_{2}}) \cdots (1-\lambda_{n_{i}}) }.
\end{equation*}
Following the same process for a point $y$, obtain an increasing sequence of natural numbers $m_{0} < m_{1} < m_{2} < \cdots $ such that $y_{m_{j}} \ne y_{m_{j} + 1}$ and corresponding point $q^{m_{j}} \in V_{m_{j}} \setminus V_{m_{j}-1}$. The extension of $u_{\eta}$ at these points is written as,
\begin{equation*}
u(q^{m_{j}}) \ \ =\ \ u_{m_{j}}(p^{m_{j}}) \ \ = \ \ \frac{u(q)}{(1-\lambda_{m_{1}}) (1-\lambda_{m_{2}}) \cdots (1-\lambda_{m_{j}}) }.
\end{equation*}
We consider the difference
\begin{equation}
\big| u(x)-u(y)  \big| \ \ \le \ \ \big| u(x)-u(p^{n_{i}}) \big|\ + \ \big| u(p^{n_{i}}) - u(q^{m_{j}})  \big| \ + \ \big| u(y) - u(q^{m_{j}})   \big|.
\end{equation}
The first and last term on the right side tend to $0$ as $i,j \rightarrow \infty $.
Since $x$ and $y$ agree on first $M_{0}$ places, there exists a position $n_{d} \le M_{0}$ for some $d \in \mathbb{N}$ such that $n_{1} = m_{1},\, n_{2} = m_{2}, \, \cdots, \, n_{d} = m_{d}$ and $n_{d+1}, m_{d+1} > M_{0}$. Then for $i, j > d$, we have

\begin{align*}
\hspace{0.5cm} &\big|u(p^{n_i}) - u(q^{m_j})  \big| \\
&  \le  \frac{\max |u_{m_0}| }{ (1-\lambda_{n_1}) (1-\lambda_{n_2}) \cdots (1-\lambda_{n_d})} \ \  \Bigg| \frac{1}{(1-\lambda_{n_{d+1}}) \cdots (1-\lambda_{n_i}) } - \frac{1}{(1-\lambda_{m_{d+1}}) \cdots (1-\lambda_{m_j})} \Bigg| \\
&\le b\, \max |u_{m_0}|\  \Bigg| \frac{1}{(1-\lambda_{n_{d+1}}) \cdots (1-\lambda_{n_i}) } - \frac{1}{(1-\lambda_{m_{d+1}}) \cdots (1-\lambda_{m_j})} \Bigg|
\end{align*}
As $x$ and $y$ come closer, $M_{0} \rightarrow \infty$, then $d,\,i,\,j\, \rightarrow \,\infty$. Therefore, the term on the right side above vanishes, proving the almost everywhere continuity of the function $u$.
\end{proof}

\section{Forbidden eigenvalues}\label{forbidden ev}
Let $m \ge 1$. We say $\lambda_{m} \in \mathbb{R}$ and $u_{m} \in \ell(V_{m}) $ are the Dirichlet eigenvalue and the corresponding Dirichlet eigenfunction of $H_{m}$ respectively, if they satisfy
\begin{eqnarray*}
H_{m} u_{m} \ \ &=& \ \ -\lambda_{m} \, u_{m} \quad \text{on } \ V_{m}\setminus V_{0}, \quad \text{subject to,} \\
u_{m}|_{V_{0}} \ \ &=& \ \ 0. 
\end{eqnarray*}
In this section, we prove that the forbidden eigenvalues as defined in section \eqref{S decimation}, play a major role in determining the Dirichlet spectrum of any difference operator $H_{m}$. We denote the Dirichlet spectrum of $H_{m}$ by $\Lambda_{m} $ and the geometric multilicity or simply, the multilplicity of the eigenvalue $\lambda_{m}$ by $M_{m}(\lambda_{m})$. Define the set $\ell_{0}(V_{m}) :=\left\lbrace f:V_{m} \longrightarrow \mathbb{R}\ \ :\ \ f|_{V_{0}} = 0\right\rbrace$. 
\medskip 

\noindent 
Let $p \in V_{m} \setminus V_{m-1}$ and $q^{1},\, q^{2},\,\cdots,\, q^{N-1} \in \mathcal{U}_{p,\,m}$. If $u_{m}$ is an eigenfunction of $H_{m}$ with the eigenvalue $\lambda_{m}$, then recall from section \eqref{S decimation}, that these points satisfy the equations \eqref{eq1} and \eqref{eq2}. We restate the $N-1$ equations as,
\begin{eqnarray}
\label{system1}
-(N-1)\,u_{m}(p)\,+\, u_{m}(q^{1})\,+\, u_{m}(q^{2})\,+\, \cdots \,+\, u_{m}(q^{N-1}) \ \ &=& \ \ -\lambda_{m}\, u_{m}(p) \nonumber\\
u_{m}(p)\, - (N-1)\, u_{m}(q^{1})\,+\, u_{m}(q^{2}) \,+\, \cdots \,+\, u_{m}(q^{N-1}) \ \ &=& \ \ -\lambda_{m}\, u_{m}(q^{1}) \nonumber\\
&\vdots & \\ 
u_{m}(p)\,+\, u_{m}(q^{1})\,+\, \cdots \,- (N-1)\, u_{m}(q^{N-2})\,+\,u_{m}(q^{N-1})\ \   &=& \ \ -\lambda_{m\,} u_{m}(q^{N-2}). \nonumber
\end{eqnarray}
Adding all the equations in the above system, we get,
\begin{equation}
\label{eq6}
(N-2) \left[ u_{m}(p)\, + \sum\limits_{j = 1}^{N-2} u_{m}(q^{j})  \right] \,+\, (N-1)\,u_{m}(q^{N-1}) \ \ =\ \ \left[ N-1- \lambda_{m} \right]\,   \left[ u_{m}(p)\, + \sum\limits_{j = 1}^{N-2} u_{m}(q^{j})  \right].
\end{equation}

\begin{proposition}
For any $m \ge 1$, $0$ is not a Dirichlet eigenvalue for any $H_{m}$.
\end{proposition}
\begin{proof}
For any $m \ge 1$, we prove that the only Dirichlet eigenfunction $u_{m} \in \ell_{0}(V_{m})$ corresponding to $\lambda_{m} = 0$, is the zero function.  Let us proceed by induction on $m$. For $m=1$, suppose that, $\lambda_{1} = 0$. Let $p, \, q^{1},\,\cdots,\, q^{N-2} \in V_{1} \setminus V_{0}$ and $q^{N-1} \in V_{0}$ as before. Thus $u_{1}(q^{N-1}) = 0$ and the equation \eqref{eq6} reduces to,
\[ u_{1}(p)\, + \sum\limits_{j\, =\, 1}^{N-2} u_{1}(q^{j}) \ \ = \ \ 0.  \]
Substituting this relation in each equation in the system \eqref{system1} above, we obtain
\[ u_{1}(p) \ \ = \ \ u_{1}(q^{1}) \ \ = \ \ \cdots \ \ = \ \ u_{1}(q^{N-2})\ \ = \ \ 0.  \]
This holds for any arbitrarily chosen $p \in V_{1}\setminus V_{0}$ and we obtain $u_{1} \equiv 0$ on $V_{1}$. Since the only eigenfunction corresponding $\lambda_{1} =0$ is the zero function, $0 \notin \Lambda_{1}$.\\

\noindent
We suppose that the statement holds for $m-1$, that is, if $\lambda_{m-1} = 0$ then $u_{m-1} \equiv 0$ on $V_{m-1}$. We prove the statement for $m$. For $\lambda_{m} = 0$, the system \eqref{system1} reduces to 
\begin{eqnarray*}
-(N-1)\,u_{m}(p)\,+\, u_{m}(q^{1})\,+\, u_{m}(q^{2})\,+\, \cdots \,+\, u_{m}(q^{N-1}) \ \ &=& \ \ 0\\
u_{m}(p)\, - (N-1)\, u_{m}(q^{1})\,+\, u_{m}(q^{2}) \,+\, \cdots \,+\, u_{m}(q^{N-1}) \ \ &=& \ \ 0\\
&\vdots & \\ 
u_{m}(p)\,+\, u_{m}(q^{1})\,+\, \cdots \,- (N-1)\, u_{m}(q^{N-2})\,+\,u_{m}(q^{N-1}) \ \  &=& \ \ 0.
\end{eqnarray*}
Solving this simultaneous system of linear equations we get
\begin{equation}
\label{eq7}
u_{m}(p) \ \ = \ \ u_{m}(q^{1}) \ \ = \ \ \cdots \ \ = \ \ u_{m}(q^{N-2})\ \ = \ \ u_{m}(q^{N-1})
\end{equation}
If the point $ p $ is chosen such that $q^{N-1} \in V_{0}$, then $u_{m}(q^{N-1}) = 0 $ and we obtain, 
 \[ u_{m}(p) \ \ = \ \ u_{m}(q^{1}) \ \ = \ \ \cdots \ \ = \ \ u_{m}(q^{N-2})\ \ = \ \ 0.  \]
Now if $q^{N-1} \in V_{m-1} \setminus V_{0}$, then using the definition of $H_{m}$ as in equation \eqref{defn_of_Hm_u} on the points in $V_{m-1}$, we write the eigenvalue equation for $q^{N-1}$ as,
\[ H_{m} u_{m} (q^{N-1}) \ \ = \ \ H_{m-1} u_{m}|_{V_{m-1}} (q^{N-1}) \, -(N-1)\,u_{m}(q^{N-1})\,+\, u_{m}(p) \,+ \sum\limits_{j\, =\, 1}^{N-2} u_{m}(q^{j})\ \ = \ \ 0. \]
Substituting the relation \eqref{eq7} in the above equation, we get
\[ H_{m-1} (u_{m}|_{V_{m-1}}) \ \ = \ \ 0 \]
By the induction hypothesis, $u_{m}|_{V_{m-1}} \equiv 0$, which implies, $u_{m} (q^{N-1}) = 0$. From equation \eqref{eq7} we conclude that $u_{m} \equiv 0$ is the only eigenfunction corrsponding to $\lambda_{m}=0$.
\end{proof}

\begin{proposition}
\label{1 & N Dev}
If $u_{m} \in \ell(V_{m})$ is an eigenfunction of $H_{m}$ corresponding to the Dirichlet eigenvalue $\lambda_{m}$, then 
\[u_{m}|_{V_{m-1}} \ \equiv \ 0 \quad \text{ if and only if } \quad \lambda_{m} \ = \ 1 \ \text{ or } \ \lambda_{m}\ =\  N. \]
\end{proposition}
\begin{proof}
The case $m=1$ is trivial. We prove the statement for any $m \ge 2$. Let $p,\,q^{1},\, q^{2},\,\cdots,\, q^{N-1} $ be as before. Suppose first that $u_{m}|_{V_{m-1}} \, \equiv \, 0 $. The equation \eqref{eq6} gives,
\[ (N-2) \left[ u_{m}(p)\, + \sum\limits_{j\, =\, 1}^{N-2} u_{m}(q^{j})  \right]  \ \ =\ \ \left[ N-1- \lambda_{m} \right]\,   \left[ u_{m}(p)\, + \sum\limits_{j\, =\, 1}^{N-2} u_{m}(q^{j})  \right]. \]
If $ \left[ u_{m}(p)\, + \sum\limits_{j\, =\, 1}^{N-2} u_{m}(q^{j})  \right] \ \ne \ 0 $, then cancelling this term on both the sides we get $\lambda_{m}\,=\,1$.\\
\noindent
And if $ \left[ u_{m}(p)\, + \sum\limits_{j\, =\, 1}^{N-2} u_{m}(q^{j})  \right] \ = \ 0 $, then substituting the relation in the first equation of the system \eqref{system1} we get $\lambda_{m}\,=\, N$.\\

\noindent
For the converse part let us first assume $\lambda_{m}\,=\,1$. Then, solving the system simultaneously we obtain the relation
\begin{equation}\label{eq08}
u_{m}(p) \ \ = \ \ u_{m}(q^{1}) \ \ = \ \ \cdots \ \ = \ \ u_{m}(q^{N-2}),
\end{equation}
which in turn gives $ u_{m}(q^{N-1})\, =\, 0$. Thus $u_{m}|_{V_{m-1}} \ \equiv \ 0$. 
\medskip

\noindent
Now, for $\lambda_{m} \,= \,N$, from the first equation in the system \eqref{system1} we get,
\begin{equation}
\label{eq8}
u_{m}(p)\,+\, u_{m}(q^{1})\,+\, u_{m}(q^{2})\,+\, \cdots \,+\, u_{m}(q^{N-1}) \ \ = \ \ 0.
\end{equation}
Observe that the eigenvalue equation on $V_{m-1}$ is written as,
\[H_{m-1} u_{m}|_{V_{m-1}} (q^{N-1}) \, -(N-1)\,u_{m}(q^{N-1})\,+\, u_{m}(p) \,+ \sum\limits_{j = 1}^{N-2} u_{m}(q^{j}) \ \ = \ \ -N\,u_{m}(q^{N-1}).   \] 
Using the relation \eqref{eq8} in above equation we obtain,
\begin{align*}
H_{m-1} u_{m-1}\ \ &=\ \ 0 \quad \text{ on } V_{m-1} \setminus V_{0},\\
u_{m}|_{V_{0}} \ \ &=\ \ 0.
\end{align*}
Therefore from the previous proposition, we conclude that $u_{m}|_{V_{m-1}}\ \equiv \ 0$. 
\end{proof}
\medskip

\noindent
It is evident from the above lemma that $1, \, N \in \Lambda_{m}$ for all $m \ge 1$. In fact, $\Lambda_{1} \, = \, \left\lbrace 1,\, N    \right\rbrace $.

\section{Dirichlet spectrum of $H_{m}$}
\label{D spectrum}
\noindent
We first verify that $1$ and $N$ are the only Dirichlet eigenvalues of $H_{1}$. Let us list the corrsponding eigenfunctions of $H_{1}$. Fix $p = (p_{1}\,\dot{p}_{2}) \in V_{1} \setminus V_{0}$ and $q^{1},\,\cdots ,\,q^{N-1}$ as before. When $\lambda_{1} = 1$, it follows from equation \eqref{eq08}, that the corresponding eigenfunction $u_{1}$ is constant on the equivalence class $ \left[ p_{1} \right]|_{V_{1} \setminus V_{0}} \ = \ \left[ p_{1} \right] \cap (V_{1} \setminus V_{0})$. Consider the characteristic function $\rchi_{[ p_{1} ]|_{V_{1} \setminus V_{0}}}\in \ell_{0}(V_{1})$ of the set $[ p_{1} ]|_{V_{1} \setminus V_{0}} $ as,
\[ \rchi_{[ p_{1} ]|_{V_{1} \setminus V_{0}}}\ \ =\ \  \begin{cases}
1\ \ \text{on} \ \ [ p_{1} ]|_{V_{1} \setminus V_{0}},\\
0\ \ \text{elsewhere.}
\end{cases}
 \]  
There are $N$ such independent functions corresponding to the equivalence classes in  $ V_{1}$. Thus $\left\lbrace  \rchi_{[ p_{1} ]|_{V_{1} \setminus V_{0}}} : \ p_{1} \in S \right\rbrace  $ forms the basis of the eigenspace for $\lambda_{1} = 1$ with $M_{1}(1) = N$. \\

\noindent
For  $\lambda_{1} = N $. define the functions $\mathcal{G}_{[p_{1}]|_{V_{1}},\, k} \in \ell_{0}(V_{1})$ with $ 1\le k \le N-2$, by
\[  \mathcal{G}_{[p_{1}]|_{V_{1}},\, k}\,(q)\ \ :=\ \ \begin{cases}
1\ \ &\text{if} \ \ q = p,\\
-1\ \ &\text{if} \ \ q = q^{k},\\
0\ \ &\text{elsewhere.}
\end{cases} \]
For a particular equivalence class $[p_{1}]|_{V_{1}} $, there are $N-2$ independent functions of the type above. It is now straightforward to verify that $ \left\lbrace \mathcal{G}_{[p_{1}]|_{V_{1}},\,k} \ : \ 1\le k \le N-2,\  p_{1} \in S \right\rbrace $ forms the basis of the eigenspace of the dimension $M_{1}(N)\, =\, N(N-2)$ for $\lambda_{1} = N$. Therefore $ \Lambda_{1} = \{1,\,N  \} $ is the complete Dirichlet spectrum of $H_{1}$, since the sum of the multiplicities of both the eigenvalues is,
\[ M_{1}(1)\,+\,M_{1}(N) \ \ = \ \ N + N(N-2) \ \ = \ \ N^{2}-N \ \ = \ \ \# (V_{1} \setminus V_{0}). \]
\medskip

\begin{proposition}\label{dimension of 1,N}
For $ m \ge 1$, the values $1,\, N $ are the Dirichlet eigenvalues of $H_{m}$ with the multiplicities given by $M_{m}(1) = N$ and $M_{m}(N) = N^{m}(N-2) $.
\end{proposition}
\begin{proof}
We have already established that $1$ and $N$ are the Dirichlet eigenvalues of $H_{m}$ and the proposition is true for $m = 1$. For $m \ge 2$, consider the points $ p \in V_{m} \setminus V_{m-1} $ and $ q^{1},\, \cdots,\,  q^{N-1} \in \mathcal{U}_{p,\,m}$ as before.\\

\noindent
Let $\lambda_{m}=1$ and $u_{m}$ be the corresponding eigenfunction. By proposition \eqref{1 & N Dev}, we have $u_{m}(q^{N-1}) = 0$. Also by equation \eqref{eq08}, we have,
\[ u_m(p)\,=\, u_m(q^1)\,=\, u_m(q^2)\,=\cdots \, u_m(q^{N-2}). \]
When $q^{N-1} \in V_{0}$, then $p=(p_{1} \,p_{2}\, \cdots \,p_{m}\, \dot{p}_{m+1}  )$ satisfies $p_{1} = p_{2} = \cdots = p_{m} \ne p_{m+1}$. For each such $p$, similar to the case when $m=1$, define the functions $ \rchi_{\left\{[ p_{1} p_{2} \cdots p_{m} ]|_{V_{m} \setminus V_{m-1}} \right\} } \in \ell_{0}(V_{m}),\ p_{1} \in S $ as follows:
\[  \rchi_{\left\{[ p_{1} p_{2} \cdots p_{m} ]|_{V_{m} \setminus V_{m-1}} \right\} }\ \ :=\ \  \begin{cases}
1\ \ \text{on} \ \ \left[ p_{1} p_{2} \cdots p_{m} \right]|_{V_{m} \setminus V_{m-1}},\\
0\ \ \text{elsewhere.}
\end{cases}
 \] 
 There are $N$ functions of this kind as $p_{1}$ varies over $S$. These functions are independent and hence belong to the eigenspace for $\lambda_{m}=1$.

\medskip

\noindent
Now, if $ q^{N-1} \in V_{m-1} \setminus V_{0}$ then the eigenvalue equation at the point $q^{N-1}$ is,
\[ H_{m} u_{m} (q^{N-1})\ \ =\ \  -\,u_{m} (q^{N-1}),  \]
which can be written using the definition of $H_{m}$ on $V_{m-1}$ as,
\[ H_{m-1} u_{m}|_{V_{m-1}} (q^{N-1})\,+\, -(N-1)\,u_{m}(q^{N-1})\,+\, u_{m}(p)\,+\, + \sum\limits_{j\, =\, 1}^{N-2} u_{m}(q^{j})      \ = \  -\,u_{m}(q^{N-1}).   \] 
Since $u_{m}|_{V_{m-1}} \equiv 0$,  we have $u_{m} (q^{N-1}) = 0 $ and $ H_{m-1} u_{m} (q^{N-1}) = 0$. Using the relation in \eqref{eq08}, we obtain $u_{m}(p)\,+\, u_{m}(q^{1})\, + \cdots +\,u_{m}(q^{N-2})\, =\, (N-1)u_{m}(p) $. Substituting for these quantities in above equations we get
\[ (N-1)\,u_m(p) = 0 \ \implies u_m(p) = 0  \]
This case does not contribute to the eigenspace for $1$. Therefore the  basis of the eigenspace for $1$ is $\left\lbrace  \rchi_{\left\{[ p_{1} p_{2} \cdots p_{m} ]|_{V_{m} \setminus V_{m-1}} \right\} }\ : \ p_{1}=\cdots=p_{m} \in S \right\rbrace   $ and hence $M_{m}(1) = N$.\\
\medskip

\noindent
Let us now find the eigenspace for $\lambda_{m}=N$. If $u_{m}$ is the corresponding eigenfunction of $H_{m}$, then again by proposition \eqref{1 & N Dev}, we have $u_{m}(q^{N-1}) = 0$. Thus equation \eqref{eq8} is rewritten as, 
\[  u_{m}(p)\,+\, u_{m}(q^{1})\,+\, u_{m}(q^{2})\, + \cdots +\,u_{m}(q^{N-2})\ \ =\ \ 0. \]
Similar to the functions $\mathcal{G}_{[p_{1}]|_{V_{1}},\,k}$ defined earlier,  for $ 1 \le k \le N-2$ the functions $\mathcal{G}_{[p_{1}\,\cdots\,p_{m}]|_{V_{m}},\, k} \in \ell_{0}(V_{m}) $ defined as, 
\[\mathcal{G}_{[p_{1}\,\cdots\,p_{m}]|_{V_{m}},\,k}\,(q)\ \ :=\ \ \begin{cases}
1\ \ &\text{if} \ \ q = p,\\
-1\ \ &\text{if} \ \ q = q^{k},\\
0\ \ &\text{elsewhere.}
\end{cases} 
 \]
are independent and belong to the basis of the eigenspace of $ N$. Note that this choice of eigenfunctions is particular to the equivalence classes of the $m$-relation, and is not dependent on the individual points of $V_{m}$. The set 
\[ \left\lbrace  \mathcal{G}_{[p_{1}\,\cdots\,p_{m}]|_{V_{m}},\,k}\ : \ 1\le k \le N-2,\ p_{1},\cdots,p_{m} \in S \right\rbrace \]
forms the basis of the eigenspace for $ N$. As there are total $N^{m}$ equivalence classes in $V_{m}$ for the $m$-relation and $N-2$ eigenfunctions corresponding to each equaivalence class, $M_{m}(N) = N^{m}(N-2)$. 
\end{proof}

\bigskip 

\noindent
For any $m \ge 2$,  the values $1$ and $N$ are not the only Dirichlet eigenvalues of $H_{m}$. For instance, the spectral decimation gives rise to four more Dirichlet eigenvalues of $H_{2}$ from $\lambda_{1}= 1$ and $\lambda_{1}= N$ which are, $\lambda_{2} = \phi_{\pm} (1),\, \phi_{\pm} (N) $. Similarly the Dirichlet eigenvalues of $H_{3}$ other than $1$ and $N$ are precisely, $ \phi_{\pm} (1),\, \phi_{\pm} (N),\, \phi_{+}\phi_{+}(1),\, \phi_{-}\phi_{+}(1),\, \phi_{-}^{2}(1),\, \phi_{+}\phi_{-}(1),\,$ $ \phi_{+}\phi_{+}(N),\, \phi_{-}\phi_{+}(N),\, \phi_{-}^{2}(N),\, \phi_{+}\phi_{-}(N) $. \\

\noindent
Suppose that $\lambda_{m} \in \Lambda_{m}$ is obtained by the spectral decimation from $\lambda_{m-1} \in \Lambda_{m-1} $ as $\lambda_{m} = \phi_{+} (\lambda_{m-1}) $ or $\phi_{-} (\lambda_{m-1}) $. If $u_{m} \in \ell_{0}(V_{m})$ and $u_{m-1} \in \ell_{0}(V_{m-1})$ are the corresponding eigenfunctions, then it is clear that $u_{m}$ is the extension of $u_{m-1}$ defined in terms of $u_{m-1}$ as in \eqref{ext to V_m}. Thus eigenspace of $\lambda_{m}$ is completely determined by the eigenspace of $\lambda_{m-1}$. Clearly, the multiplicities of $\lambda_{m}$ and $\lambda_{m-1}$ are the same. 
\medskip 

\noindent 
Also observe the following important fact.  Under the Dirichlet boundary conditions, the eigenfunctions take the value zero on the boundary $V_{0}$. Therefore the grand total of the multiplicities of the Dirichlet eigenvalues of $H_{m}$ should be $N^{m+1}-N$, which is the cardinality of the set $V_{m} \setminus V_{0}$.

\begin{theorem}
For $m \ge 2,$ the complete Dirichlet spectrum of $H_{m}$ is given by
\[ \Lambda_{m} \ = \ \left\lbrace 1,\, N,\, \phi_{+}(\lambda_{m-1}),\, \phi_{-}(\lambda_{m-1})\ :\ \lambda_{m-1} \in \Lambda_{m-1}  \right\rbrace.  \]
\end{theorem}
\begin{proof}
We rely on the multiplicities count to determine the complete spectrum. The inclusion 
\[ \left\lbrace 1,\, N,\, \phi_{+}(\lambda_{m-1}),\, \phi_{-}(\lambda_{m-1})\ :\ \lambda_{m-1} \in \Lambda_{m-1}  \right\rbrace \ \subset  \ \Lambda_{m} \]
is trivial. Let us now proceed by induction on $m$. Let $m=2$. By the proposition \eqref{dimension of 1,N}, we know that
\[ M_{2}(1) \ =\  N\ \ \text{and }\ \ M_{2}(N)\  =\  N^{2}(N-2).  \]
Since
\begin{eqnarray*}
M_{2}(\phi_{+}(1)) \ \ &= \ \ M_{2}(\phi_{-}(1)) \ \ &= \ \ M_{1}(1) \\
M_{2}(\phi_{+}(N)) \ \ &= \ \ M_{2}(\phi_{-}(N)) \ \ &= \ \ M_{1}(N),
\end{eqnarray*}
the grand total of the multiplicities of these Dirichlet eigenvalues is,
\[ M_{2}(1)\, + \,M_{2}(N)\, +\, 2\,M_{1}(1) \, + \, 2\, M_{1}(N) \ = \ N^{3}-N \ = \ \# (V_{2}\setminus V_{0} ) \]
as expected. Thus $\left\lbrace 1,\, N,\, \phi_{+}(\lambda_{1}),\, \phi_{-}(\lambda_{1})\ :\ \lambda_{1} \in \Lambda_{1}  \right\rbrace $ is the complete spectrum of $H_{2}$. \\

\noindent 
Assuming the statement holds for $m-1$, we prove for $m$. The grand total of the multiplicities of $\lambda_{m-1} \in \Lambda_{m-1}$ is,
\[ \sum\limits_{\lambda_{m-1}\, \in\, \Lambda_{m-1} } M_{m-1}(\lambda_{m-1}) \ = \ N^{m} - N. \]
Let us now count the multiplicities of the already known Dirichlet eigenvalues of $H_{m}$. The total of the multiplicities of the eigenvalues in the set $\left\lbrace 1,\, N,\, \phi_{+}(\lambda_{m-1}),\, \phi_{-}(\lambda_{m-1})\, :\, \lambda_{m-1} \in \Lambda_{m-1}  \right\rbrace  $ is,
\[ M_{m}(1) \,+\, M_{m}(N) \,+ \sum\limits_{\lambda_{m-1}\, \in\, \Lambda_{m-1}} M_{m}\left( \phi_{+}(\lambda_{m-1}) \right) \,+ \sum\limits_{\lambda_{m-1}\, \in\, \Lambda_{m-1}} M_{m}\left( \phi_{-}(\lambda_{m-1}) \right) \]
Using proposition \ref{dimension of 1,N} and the fact that $ M_{m}\left( \phi_{+}(\lambda_{m-1}) \right) \ = \ M_{m-1}(\lambda_{m-1})$, this total is,
\begin{align*}
N\,+\, N^{m}\,(N-2)\,+\, 2\sum\limits_{\lambda_{m-1}\, \in\, \Lambda_{m-1} } M_{m-1}(\lambda_{m-1}) \ &=\ N\,+\, N^{m}\,(N-2) \,+\, 2\left( N^{m}\,- \,N\right) \\
&= \ N^{m+1} \,-\, N \ \ = \ \ \# (V_{m} \setminus V_{0}).
\end{align*}
Therefore $ \Lambda_{m} \ = \ \left\lbrace 1,\, N,\, \phi_{+}(\lambda_{m-1}),\, \phi_{-}(\lambda_{m-1})\ :\ \lambda_{m-1} \in \Lambda_{m-1}  \right\rbrace $ is the complete spectrum of $H_{m}$.
\end{proof}

\section{Dirichlet eigenvalues of $\Delta$}
\label{main theorem}
Before proceeding to the main theorem of the paper, we bring the readers' attention to the fact, that although the following theorem determines the Dirichlet eigenvalues of the Laplacian $\Delta$, it does not guarantee that these eigenvalues constitute the complete Dirichlet spectrum. The question that still remains open, as was mentioned in the passing in the introduction, is the following. 
\medskip 

\noindent 
Suppose $\lambda$ is a Dirichlet eigenvalue and $u$ is the corresponding eigenfunction of $\Delta$. Then, is it possible to find a threshold $m_{0} \in \mathbb{Z}_{+}$ and sequences $\{ \lambda_{m} \}_{m\, \ge\, m_{0}},\ \{ u_{m} \}_{m\, \ge\, m_{0}}$ and $\{ H_{m} \}_{m\, \ge\, m_{0}}$ such that 
\begin{enumerate} 
\item $\lambda = \lim\limits_{m\, \to\, \infty} N^{m + 1} \lambda_{m}$; 
\item $u_{m} = u|_{V_{m}}$; 
\item $\lambda_{m}$ is an eigenvalue of $H_{m}$ with corresponding eigenfunction $u_{m}$; 
\item the sequence of eigenvalues and eigenfunctions namely $\{ \lambda_{m} \}_{m\, \ge\, m_{0}}$ and $\{ u_{m} \}_{m\, \ge\, m_{0}}$ satisfy the iterative relations in \eqref{lambda2} and \eqref{ext_efunction} respectively. 
\end{enumerate} 

\noindent
We conclude the paper with the following theorem that gives a Dirichlet eigenvalue and its corresponding eigenfunction of the Laplacian.
\begin{theorem}
Let $\lambda_{m_{0}}$ be a Dirichlet eigenvalue and $u_{m_{0}} \in \ell (V_{m_{0}})$ be the corresponding eigenfunction of $H_{m_{0}}$, for some $m_{0} \ge 1$. For $m \ge m_{0}$, suppose $\lambda_{m} := \phi_{-}^{m-m_{0}} (\lambda_{m_{0}})$ is a Dirichlet eigenvalue of $H_{m}$ with the corresponding eigenfunction $u_{m} \in \ell(V_{m})$, obtained by the spectral decimation method, as given by equations \eqref{lambda2} and \eqref{ext_efunction}, respectively. Then 
\[ \lambda\ \ :=\ \ \lim\limits_{m\, \to\, \infty} N^{m+1} \lambda_{m} \] 
is a Dirichlet eigenvalue of $\Delta$ with corresponding eigenfunction $u$ as defined in equation \eqref{ext full}.
\end{theorem}
\begin{proof}
As discussed in section \eqref{cts ext}, the sequence $\{\lambda_{m} \}_{m\, \ge\, m_{0}}$ converges to zero atleast as fast as $\{\frac{1}{N^{m+1}} \}_{m\, \ge\, m_{0}}$. Thus, the limit $\lambda := \lim\limits_{m\, \to\, \infty} N^{m+1} \lambda_{m}  $ exists. Further, suppose $u$ is as defined in equation \eqref{ext full} with $u_{m} = u|_{V_{m}}$, where $u_{m}$ is an eigenfunction corresponding to the eigenvalue $\lambda_{m}$, for $m \ge m_{0}$. Then we obtain,
\begin{eqnarray*}
\Delta u \ \ & = & \ \ \lim\limits_{m \, \to\, \infty} N^{m+1} H_{m} u_{m} \\
& = & \ \ \lim\limits_{m \, \to \, \infty}- N^{m+1} \lambda_{m}\, u_{m} \\
& = & \ \ -\lambda \, u.
\end{eqnarray*}
Thus we conclude that $\lambda$ is a Dirichlet eigenvalue of $\Delta$.
\end{proof}


\begin{thebibliography}{00}


%

\bibitem{bks}
Bedford, T., Keane, M. and Series, C., (eds.), \textit{Ergodic theory, symbolic dynamics, and hyperbolic spaces}, Oxford Science Publications, The Clarendon Press, Oxford University Press, New York, 1991. 

\bibitem{bmn}
Blanchard, F., Maass, A. and Nogueira, A., (eds.), 2000. \textit{Topics in symbolic dynamics and applications}, London Mathematical Society Lecture Note Series, vol. 279, Cambridge University Press, Cambridge, 2000.

\bibitem{bowen1} 
Bowen, R., ``Markov partitions for Axiom A diffeomorphisms", \textit{Amer. J. Math.}, 92 (1970), 725 - 747.

\bibitem{bowen2} 
Bowen, R., ``Markov partitions and minimal sets for Axiom A diffeomorphisms", \textit{Amer. J. Math.} 92 (1970), 907 - 918. 

\bibitem{bowen3} 
Bowen, R., ``Symbolic dynamics for hyperbolic flows", \textit{Amer. J. Math.} 95 (1973), 429 - 460.


\bibitem{fuku_shima}
Fukushima, M. and Shima, T., ``On a spectral analysis for the Sierpi\'{n}ski gasket", \textit{Potential Anal.} 1 (1992), no. 1, 1 - 35. 


\bibitem{hadamard}
Hadamard, J., ``Les surfaces à courbures opposées et leurs lignes géodésique", \textit{J. Math. pures appl.}, 4 (1898), 27-73.
%
%
%
%

\bibitem{kigamimetric}
Kigami, J., ``Effective resistances for harmonic structures on p.c.f. self-similar sets", \textit{Math. Proc. Cambridge Philos. Soc.} 115 (1994), no. 2, 291 - 303. 


%
%

\bibitem{kitchens}
Kitchens, B.P., \textit{Symbolic dynamics: One-sided, two-sided and countable state Markov shifts}, Universitext, Springer-Verlag, Berlin, 1998.   
%

\bibitem{LM}
Lind, D., Marcus, B. \textit{An introduction to symbolic dynamics and coding}, Cambridge University Press, Cambridge, 1995.

\bibitem{lyubich} 
Lyubich, M.Yu., ``The dynamics of rational transformations: the topological picture", \textit{Russian Math. Surveys}, 41 (1986), no.4, 43 - 117.

\bibitem{MH}
Morse, M., Hedlund, G.A., ``Symbolic Dynamics", \textit{Amer. J. Math.}, 60 (1938), no. 4, 815 - 866.


\bibitem{parrypolli}
Parry, W. and Pollicott, M., ``Zeta functions and the periodic orbit structure of hyperbolic dynamics", \textit{Ast\'{e}risque} 187(1990), no. 188, pp. 1 - 268.
%
\bibitem{ramtoul}
Rammal, R. and Toulouse, G., ``Random walks on fractal structures and percolation clusters", \textit{ Journal de Physique Lettres}, 44(1983), no. 1, pp. L13 - L22.



\bibitem{rammal}
Rammal, R., ``Spectrum of harmonic excitations on fractals", \textit{Journal de Physique}, 45 (1984), no. 2, pp. 191 - 206.

\bibitem{shima}
Shima, T., ``On eigenvalue problems for the random walks on the Sierpi\'{n}ski pre-gaskets", \textit{Japan J. Indust. Appl. Math.} 8 (1991), no.1, 127 - 141. 

\bibitem{arxiv}
Sridharan, S. and Tikekar, S.N., ``An analogue of the Dirichlet boundary value problem on the shift space", arXiv e-print (2019), arXiv:1907.09139.


\end{thebibliography}
\end{document}